\documentclass[11pt]{article}
\usepackage[T1]{fontenc}

\usepackage{mathtools}
\usepackage{amssymb}
\usepackage{amsthm}
\usepackage{stmaryrd}
\usepackage{graphicx}
\usepackage{tikz}
\usepackage{tikz-cd} 
\usepackage{subcaption}

\usepackage{xcolor}
\usepackage{hyperref}    

\newcommand{\Mde}{\mathbf{Mde}}
\newcommand{\Scl}{\mathbf{Scl}}

\newcommand{\T}{\mathrm{T}}

\newcommand{\Z}{\mathbb{Z}}

\theoremstyle{definition}
\newtheorem{definition}{Definition}[section]
\newtheorem{example}{Example}[section]

\theoremstyle{plain}
\newtheorem{proposition}{Proposition}[section]
\newtheorem{theorem}{Theorem}[section]
\newtheorem{corollary}{Corollary}[section]

\usepackage{biblatex}
\begin{document}
\title{A Theory of Scales and Orbit Covers}
\author{Drew Flieder\\
Independent Scholar\\
\texttt{drew@drewflieder.xyz}}

\date{}

\maketitle            

\begin{abstract}
This paper develops a formal theory of musical scales and their harmonic coverings and introduces \emph{orbit covers}: coverings obtained by translating a fixed subset across a scale via a group action. Orbit covers generalize familiar constructions, such as the covering of the diatonic scale by tertian triads, and are motivated by the search for a generalized harmonic framework extending common-practice tonality.

We model modes as group structures associated with pitch-class sets and scales as torsors, introducing \emph{scale covers} and, in particular, \emph{orbit covers}. To each orbit cover we associate a \emph{nerve complex} encoding its intersection structure and associated topological invariants. We classify triadic orbit covers of heptatonic scales up to affine symmetry and nerve isomorphism.

These results support a broader theory of harmonic organization with analytical and compositional applications.

\end{abstract}

\noindent\textbf{Note.}
This is a preprint of a paper submitted to the proceedings of the 10th Conference on Mathematics and Computation in Music (MCM 2026). This version has not undergone peer review. The final Version of Record will be available via Springer Nature.

\section{Introduction}

This paper develops a formal theory of scales and their harmonic coverings, using tools from algebra and combinatorial topology to classify pitch-class collections based on how they are covered by subsets. The motivation, both analytical and creative, is to lay the groundwork for a new system of harmony that generalizes the functional harmony of common-practice tonality, while allowing the greater harmonic flexibility often found in post-tonal music. This project is rooted in the author's compositional practice, which seeks to extend tonal function beyond diatonic harmony.

Central to this theory is the idea that harmonic structure arises from the relations between subsets of a scale. To model this, we introduce a class of structures called \emph{orbit covers}, which represent chordal coverings of a scale generated by iterated scalar translations of an initial subset. The diatonic scale covered by tertian triads serves as a canonical example of such a structure.

Section~\ref{sec:modes-and-scales} formalizes the notion of a \emph{mode} as a group structure on a pitch-class collection, following which we define a \emph{scale} as a torsor over a cyclic group. This construction serves as the foundation for the structures developed in the sections that follow. A recent and complementary approach is given by \cite{harasim2022axiomatic}, who provide a rigorous axiomatic account of scales as cyclic embeddings between finite cyclically ordered sets. Although the present work does not build on that framework, it shares a related aim of giving a structural foundation for scale theory within a mathematical setting. We do not, however, engage with the detailed properties of particular scale types, for which there exists a rich and well-developed literature (see, e.g., \cite{carey1989aspects,clampitt2011modes,clough1985variety,clough1986musical,clough1991maximally,clough1999scales,noll2015triads}, among many others).

In Section~\ref{sec:scale-covers}, we define \emph{scale covers} as families of subsets that exhaust a scale, and specialize to \emph{orbit covers}, with particular focus on \emph{primitive orbit covers}—those generated by a single chord type under translation. We obtain a classification result in Theorem~\ref{thm:triadic-classes-seven} regarding the number of distinct triadic orbit covers of seven-note scales.

In Section~\ref{sec:topology-of-scale-covers}, we investigate the topological structure of orbit covers via their associated \emph{nerve complexes}. These encode the intersection patterns among chords in a cover, revealing combinatorial invariants such as connectedness and the presence of $n$-dimensional holes. We show that orbit covers can be classified up to isomorphism by affine automorphisms of the scale's torsor group. This yields a refined classification result in Theorem~\ref{thm:heptatonic-triadic-isoclasses}, enumerating the isomorphism classes of triadic orbit covers of seven-note scales.

The theory presented here forms the foundational layer of a broader framework the author is developing in an unpublished manuscript titled \emph{Mathematical Principles of a Generalized Tonal Practice}, which aims to generalize functional harmony across arbitrary pitch-class collections. A brief outlook on this larger system and its compositional implications appears in the concluding section; a more comprehensive treatment is developed in my companion paper, ``Functional Harmony After Tonality: Composing with Orbit Covers.''

\section{Modes and Scales}
\label{sec:modes-and-scales}

Musical scales are traditionally conceived as ordered pitch collections, often associated with concepts such as stepwise adjacency, directionality, and modality. In this section, we develop a formal apparatus for encoding such structures by treating scales as torsors over cyclic groups and defining their modes as group structures with distinguished tonics. This prepares the ground for our later development of harmonic coverings—particularly orbit covers—which rely on this stepwise structure to generate chords across the scale.

\subsection{Modes}

To give a preliminary account of a scale, let $X \subseteq \Z_N$ be a subset of the $N$-tone pitch-class universe. For example, the C major scale is typically represented as
\[
X = \{ 0, 2, 4, 5, 7, 9, 11 \} \subseteq  \Z_{12}.
\]

This bare set-theoretic notion, however, omits the structural features usually associated with scales, such as adjacency and directionality. For instance, given $x, y \in X$, we may wish to ask whether $y$ is a scalar step above $x$, or how many steps apart they are. These questions cannot be answered if $X$ is treated merely as a set. Additional structure is needed.

A first observation is that a scale $X$ appears to carry a cyclic structure, and hence may be modeled by the group $\Z_n$, where $n = |X|$. Thus, we seek a bijection
\[
g : X \longrightarrow \Z_n
\]
that assigns to each $x \in X$ a position in a stepwise group. This allows us to speak meaningfully about adjacency: $y$ lies $k$ steps above $x$ precisely when $g(y) - g(x) = k$.

Clearly, not every bijection reflects the conventional scalar order. For example, if $X$ is the C major scale and $h : X \to \Z_7$ sends $2 \mapsto 0$ and $11 \mapsto 1$, then $h(11) - h(2) = 1$, implying that B$\natural$ is adjacent to D$\natural$, which is a clearly atypical assignment. While such maps may have compositional interest, they do not represent the standard step structure. We thus seek a canonical class of bijections $g : X \to \Z_n$ that preserve scalar adjacency.

Although listing the elements of $X$ in ascending order provides an initial ordering, this is arbitrary up to rotation. To select a canonical representative, we apply the \emph{normal order}, which chooses the most compact rotation of $X$, minimizing the span of intervals from left to right. This yields a set map
\begin{equation}
\mathrm{norm}_X : [n] \longrightarrow X,
\end{equation}
where $[n] = {0,1,\ldots,n-1}$ indexes the elements of $X$ in normal order.

Identifying $[n]$ with $\Z_n$ via $\eta : [n] \to \Z_n$, $i \mapsto i$, we canonically label the elements of $X$ with group coordinates. This induces a map $\mu : X \to \Z_n$ such that the following diagram commutes:
\[\begin{tikzcd}
	{[n]} && X \\
	\\
	&& {\Z_n}
	\arrow["{\mathrm{norm}_X}", from=1-1, to=1-3]
	\arrow["\eta"', from=1-1, to=3-3]
	\arrow["\mu", from=1-3, to=3-3]
\end{tikzcd}\]

Using $\mu$, we may define a group operation on $X$ by
\[
x \oplus y \coloneqq \mu^{-1}\big(\mu(x) + \mu(y)\big),
\]
which equips $X$ with a group structure isomorphic to $\mathbb Z_n$.

We are now in a position to define a notion of \emph{mode}.

\begin{definition}[Mode]
\label{def:mode}
Let $\Z_N$ be a cyclic group representing the $N$-element pitch-class universe, and let $X \subseteq \Z_N$ with $|X| = n$. For each $i \in \Z_n$, let $\T_i : \Z_n \to \Z_n$ denote translation by $i$. Define $\mu_i : X \to \Z_n$ by
\[
\mu_i \coloneqq \T_i \circ \mu. 
\]
We define the \emph{$(i+1)$th mode of $X$} as the group $(X, \oplus_{\mu_i})$, where
\[
x \oplus_{\mu_i} y \coloneqq {\mu_i}^{-1}({\mu_i}(x) + {\mu_i}(y)).
\]
We call the identity element $t \in X$ the \emph{tonic} and refer to $\Z_n$ as the \emph{degree group} of the mode.\footnote{Formally, this construction may be expressed via a type-forming rule
\[
\Z_N : \mathsf{Cyc}, \; X : P(\Z_N), \; i : [n] \;\vdash\; \mathsf{mode}(\Z_N, X, i) : \mathsf{Mode},
\]
as presented in \cite{flieder2025typetheory}, where $\mathsf{Cyc}$ denotes a type of cyclic groups and $P(\Z_N)$ the type of subsets of $\Z_N$. However, a type-theoretic formulation is beyond our present scope.}
\end{definition}

\begin{example}[Modes of a Major Scale]
Let \( X = \{0,2,4,5,7,9,11\} \subset \Z_{12} \) denote the C major scale. Its normal order is
\[
\mathrm{norm}_X = (11, 0, 2, 4, 5, 7, 9),
\]
and thus its first mode \( \mu_0 : X \to \Z_7 \) is \emph{Locrian}, since it assigns $11$ (the B$\natural$) to the tonic position. 
\end{example}

We will often denote a mode $(X, \oplus_{\mu_i})$ simply by its underlying set $X$, omitting the subscript and writing the operation as $\oplus$ when no confusion arises. 

\begin{definition}[Mode Homomorphism]
\label{def:mode-homomorphism}
Let $(X, \oplus_{\mu_i})$ and $(X', \oplus_{\mu_{i'}})$ be modes with $|X| = n$ and $|X'| = n'$. A \emph{mode homomorphism}
\[
\phi : (X, \oplus_{\mu_i}) \longrightarrow (X', \oplus_{\mu_{i'}})
\]
is a function $\phi : X \to X'$ induced by a group homomorphism
\[
\hat{\phi} : \Z_n \longrightarrow \Z_{n'}
\]
that makes the following diagram commute:
\[\begin{tikzcd}
	X && {X'} \\
	\\
	{\Z_n} && {\Z_{n'}}
	\arrow["\phi", from=1-1, to=1-3]
	\arrow["{\mu_i}"', from=1-1, to=3-1]
	\arrow["{\mu_{i'}}", from=1-3, to=3-3]
	\arrow["{\hat{\phi}}"', from=3-1, to=3-3]
\end{tikzcd}\]
The homomorphism $\hat{\phi}$ is of the form
\begin{equation}
\label{eqn:hom-condition1}
\hat{\phi}(j) = aj \bmod n',
\end{equation}
where 
\begin{equation}
\label{eqn:hom-condition2}
a = \frac{n'}{\gcd(n, n')}.
\end{equation}
These cases ensure that the ``clockwise'' orientation of the mode is preserved.
\end{definition}

Mode homomorphisms compose associatively and admit identities, so the class of all modes and their homomorphisms forms a category, denoted $\Mde$.

\subsection{Scales}

While it is formally convenient to model a mode as a group structure on a pitch-class set, this framing is somewhat artificial. In practice, one is not usually interested in adding scale degrees; rather, one cares about the distances—or step relations—between them. The group identity in this context functions mainly to designate a distinguished tonic, providing an ``origin'' for the scale. The binary operation $\oplus_{\mu_i}$ then allows us to speak of stepwise motion, typically via subtraction.

If we ``forget the tonic'' $t \in X$ in a mode $(X, \oplus_{\mu_i})$, we are left with a structure known as a \emph{torsor}. Informally, a torsor is like a group without a specified identity element. This more accurately reflects the way scalar relationships are typically understood, namely, as origin-independent. Indeed, chromatic pitch-class space is better modeled as a torsor rather than a group, since no pitch class functions intrinsically as an identity.\footnote{Likewise for chromatic pitch space; see \cite{baez2009torsors}.} When we speak of pitch-class relations, we are describing intervals as relative differences, not as coordinates anchored to a fixed reference point.

\begin{definition}[Torsor]
Let $G$ be a group. A \emph{$G$-torsor} is a set $X$ equipped with a simply transitive action $\tau : G \times X \to X$. 
\end{definition}

In our setting, any mode $(X, \oplus_{\mu_i})$ with $|X| = n$ induces a $\Z_n$-torsor:
\begin{equation}
\label{eqn:torsor}
\begin{matrix}
\tau : & \Z_n \times X &  \longrightarrow &  X \\
&  \tau(g, x) & \longmapsto &  \mu_i^{-1}(g + \mu_i(x)),
\end{matrix}
\end{equation}
which is in fact the same torsor for every mode on $X$. The action $\tau$ expresses how pitch classes in $X$ are related by their stepwise distances: $\tau(g, x)$ yields the pitch class that lies $g$ steps above $x$.

This leads naturally to the definition of a \emph{scale}.

\begin{definition}[Scale]
Let $(X, \oplus_{\mu_i})$ be a mode with $|X| = n$. The \emph{(underlying) scale} of $X$ is the $\Z_n$-torsor defined in~\eqref{eqn:torsor}; it forgets the tonic of $X$ while retaining its stepwise structure. We denote this scale by $(X, \Z_n)$, or simply by $X$ when unambiguous.
\end{definition}

To define morphisms between scales, we recall the general notion of a torsor homomorphism. A \emph{torsor homomorphism} between a $G$-torsor $(X, G)$ and an $H$-torsor $(Y, H)$ is a pair $(f, \varphi)$ consisting of a set function $f : X \to Y$ and a group homomorphism $\varphi : G \to H$ such that
\[
f(g \cdot x) = \varphi(g) \cdot f(x)
\]
for all $g \in G$ and all $x \in X$.

This yields the corresponding notion for scales.

\begin{definition}[Scale Homomorphism]
\label{def:scale-homomorphism}
Let $(X, \Z_n)$ and $(X', \Z_{n'})$ be scales. A \emph{scale homomorphism} is a pair 
\[
(\phi_{i,i'}, \hat{\phi}) : (X, \Z_n) \longrightarrow (X', \Z_{n'}),
\] 
where $\hat{\phi} : \Z_n \to \Z_{n'}$ is the group homomorphism from Definition~\ref{def:mode-homomorphism} and $\phi_{i,i'} : X \to X'$ is a set map such that the diagram
\[\begin{tikzcd}
	X && {X'} \\
	\\
	{\Z_n} && {\Z_{n'}}
	\arrow["{\phi_{i,i'}}", from=1-1, to=1-3]
	\arrow["{\mu_i}"', from=1-1, to=3-1]
	\arrow["{\mu_{i'}}", from=1-3, to=3-3]
	\arrow["{\hat{\phi}}"', from=3-1, to=3-3]
\end{tikzcd}\]
commutes. Since any choice of modes $(\mu_i, \mu_{i'})$ determines the same underlying torsors, these maps define all scale homomorphisms from $(X, \Z_n)$ to $(X', \Z_{n'})$.
\end{definition}

Scales and their homomorphisms thus form a category, denoted $\Scl$.

\section{Scale Covers and Their Classification}
\label{sec:scale-covers}

When composing with a particular scale $X$, one often seeks a collection of harmonic subsets from which to construct chords. These subsets collectively serve as the harmonic material available within the scale. To formalize this idea, we introduce the notion of a \emph{scale cover}, a covering of the scale by subsets. We then define a special class of covers called \emph{orbit covers}, which arise from the action of scalar translation on a generating subset. This leads to a classification result in Section \ref{sub:classification-of-orbit-covers}.

\subsection{Scale Covers and Morphisms}

\begin{definition}[Scale Cover]
\label{def:scale-cover}
Let $X$ be a scale (i.e., a $\mathbb{Z}_n$-torsor). A \emph{scale cover} is a family $\{U_i\}_{i \in I}$ of subsets of $X$ such that $\bigcup_{i \in I} U_i = X$. 

A \emph{morphism of scale covers}
\[
(f, \varphi) : (X, \{U_i\}) \longrightarrow (Y, \{V_j\})
\]
consists of a scale homomorphism $f : X \to Y$ and an index map $\varphi : I \to J$ such that
\[
f(U_i) \subseteq V_{\varphi(i)} \quad \text{for all } i \in I.
\]
These form the category $\mathbf{SclCov}$, with a canonical forgetful functor to the category of scales $\mathbf{Scl}$.
\end{definition}

This general notion captures any family of chords that exhausts a scale. However, musical practice often employs such families that are not arbitrary, but instead exhibit systematic generation from a single chord. For example, diatonic triads arise by translating a three-note chord uniformly across the scale. This motivates the more specialized concept of an \emph{orbit cover}.

\subsection{Orbit Covers and Primitive Covers}

\begin{definition}[Orbit Cover]
\label{def:orbit-cover}
Let $X \subseteq \mathbb{Z}_N$ be a scale with $|X| = n$, and let $A \subset X$ be nonempty. Let $\tau : \mathbb{Z}_n \times X \to X$ denote the torsor action associated with $X$.

To emphasize its musical interpretation, we write the elements of $\mathbb{Z}_n$ as $\T_i$, so that $\T_i(x) \coloneqq \tau(i, x)$ represents scalar translation upward by $i$ steps. This defines the group of scalar translations
\[
\T X = \{\, \T_i : X \to X \mid i \in \mathbb{Z}_n \,\} \cong \mathbb{Z}_n,
\]
which acts transitively on $X$.

The \emph{orbit cover} generated by $A$ is the family
\[
X^{(A)} = \{\, \T_i(A) \mid i \in \mathbb{Z}_n \,\}.
\]
If $\gcd(n, |A|) = 1$, then the action of $\T X$ on $X^{(A)}$ is free, and the cover is said to be \emph{primitive}. In this case, we have 
\[
|X^{(A)}| = |X|.
\]
\end{definition}

Primitive orbit covers are particularly salient because each chord in the cover corresponds uniquely to an element of the scale. This is exemplified in the case of diatonic triadic harmony, where each triad is indexed by its root and thus in bijection with the scale degrees.

\begin{example}[Diatonic Triadic Orbit Cover]
Consider the C major scale
\[
X = \{0, 2, 4, 5, 7, 9, 11\} \subseteq \mathbb{Z}_{12},
\]
equipped with its usual scalar step ordering. This defines a $\mathbb{Z}_7$-torsor under the action of scalar translation.

Let $A = \{0, 4, 7\} = \{\mathrm{C}, \mathrm{E}, \mathrm{G}\}$, the tonic triad. Applying scalar translation $\T_i$ for $i \in \mathbb{Z}_7$ generates the orbit
\[
X^{(A)} = \{ \T_i(A) \mid i \in \mathbb{Z}_7 \} = \{\text{C maj}, \text{D min}, \text{E min}, \text{F maj}, \text{G maj}, \text{A min}, \text{B dim}\},
\]
which is precisely the set of diatonic triads. Since $\gcd(7,3) = 1$, the cover is primitive, and each triad corresponds uniquely to a scale degree.
\end{example}

\subsection{Classification of Orbit Covers}
\label{sub:classification-of-orbit-covers}

We now seek a classification of the orbit covers that may arise for a given scale $X \subseteq \Z_N$ with $|X| = n$. We begin by outlining the general formalism for such covers; the specific case of triadic orbit covers over seven-note scales is addressed in Theorem \ref{thm:triadic-classes-seven}.

Our first step is to describe all possible $k$-element subsets of $X$---which we call \emph{$k$-ads}---up to scalar translation. To this end, we proceed as follows:\begin{enumerate}
\item Choose a basepoint $x \in X$.
\item Specify an ordered sequence
	\[
	\sigma = (i_1, \ldots, i_k)
	\]
	of positive integers summing to $n$, i.e.,
	\begin{equation}
	\label{eqn:sigma-n-k}
	\sigma \in \Sigma(n, k) \coloneqq 
	\left\{\, (i_1, \ldots, i_k) \in \mathbb{Z}_{>0}^k 
	\ \middle| \ 
	\sum_{j=1}^k i_j = n \,\right\}.
	\end{equation}
\item Let $S_b = \sum_{j=1}^b i_j$ denote the $b$th partial sum of $\sigma$, with $S_0 = 0$. Define the corresponding $k$-ad by
	\begin{equation}
	\label{eqn:sigma-x}
	\sigma(x) \coloneqq \left\{\, \T_{S_b}(x) \;\middle|\; 0 \leq b < k \,\right\} \subset X.
	\end{equation}
\end{enumerate}

We refer to elements of $\Sigma(n,k)$ as \emph{interval compositions} of $n$ into $k$ parts. Each $\sigma \in \Sigma(n,k)$ encodes a way of partitioning the scale span into $k$ successive intervals. The set $\sigma(x)$ is called the \emph{realization} of $\sigma$ at $x$, or the \emph{chord} generated by $\sigma$ at $x$.

We now introduce an equivalence relation on interval compositions: two compositions $\sigma, \sigma' \in \Sigma(n,k)$ are considered equivalent, written $\sigma \sim \sigma'$, if there exists a scalar translation $\T_p \in \T X$ such that
\[
\T_p(\sigma(x)) = \sigma'(x).
\]

The following proposition characterizes this equivalence precisely in terms of coordinate rotation.

\begin{proposition}
\label{prop:unique-interval-sequence}
Let $\sigma, \sigma' \in \Sigma(n,k)$. Then $\T_p(\sigma(x)) = \sigma'(x)$ for some $p \in \mathbb{Z}_n$ if and only if $\sigma'$ is a rotation of $\sigma$, i.e., there exists $\mathrm{R}_i \in \mathbf{R}\mathbb{Z}^k$ such that
\[
\mathrm{R}_i(\sigma) = \sigma'.
\]
\end{proposition}

\begin{proof}
For the forward direction, let $\mathrm{R}_i(\sigma)$ denote the $i$th rotation of $\sigma$:
\[
\mathrm{R}_i(\sigma) = (\sigma_i, \sigma_{i+1}, \ldots, \sigma_k, \sigma_1, \ldots, \sigma_{i-1}).
\]
Let $S_i = \sum_{j=1}^i \sigma_j$ denote the $i$th partial sum. Then
\[
\mathrm{R}_i(\sigma)(x) = \T_{-S_i}(\sigma(x)),
\]
since shifting the starting point of the interval cycle by $-S_i$ yields the rotated realization.

Conversely, for a fixed $\sigma(x)$, there are exactly $k$ transpositions that map $x$ to one of its $k$ elements. Each such transposition corresponds uniquely to a rotation of $\sigma$, since $S_i \neq S_j$ for $i \neq j$. Thus no other transpositions yield a realization not accounted for by some $\mathrm{R}_i(\sigma)$. 
\end{proof}

It is a standard combinatorial fact (e.g., \cite[p.~44]{flajolet2009analytic}) that
\[
|\Sigma(n,k)| = \binom{n-1}{k-1},
\]
i.e., there are $\binom{n-1}{k-1}$ compositions of $n$ into $k$ positive parts. By Proposition~\ref{prop:unique-interval-sequence}, each orbit cover corresponds uniquely to a rotation class
\[
[\sigma] = \{ \mathrm{R}_i(\sigma) \mid 0 \le i < k \} \subset \Sigma(n,k),
\]
and thus the number of distinct orbit covers of type $(n,k)$ is the number of such classes.

Each equivalence class $[\sigma]$ corresponds to a unique orbit of $k$-ads with respect to the action of $\T X$, and together these classes exhaust all possible $k$-adic orbit covers of $X$. Given such a class representative $\sigma$, we denote the corresponding cover of $X$ by $X^{(\sigma)}$.

\begin{theorem}[Triadic Interval Classes for Heptatonic Scales]
\label{thm:triadic-classes-seven}
Under the action of the rotation group $\mathbf{R}\mathbb{Z}^3$ on $\Sigma(7,3)$, there are exactly five rotation classes:
\[
\begin{aligned}
[(1,1,5)] &= \{(1,1,5), (1,5,1), (5,1,1)\}, \\
[(1,2,4)] &= \{(1,2,4), (2,4,1), (4,1,2)\}, \\
[(1,3,3)] &= \{(1,3,3), (3,1,3), (3,3,1)\}, \\
[(1,4,2)] &= \{(1,4,2), (2,1,4), (4,2,1)\}, \\
[(2,2,3)] &= \{(2,2,3), (2,3,2), (3,2,2)\}.
\end{aligned}
\]
\end{theorem}

\begin{proof}
This follows from exhaustive enumeration of $\Sigma(7,3)$ and identification of cyclic rotations. 
\end{proof}

\begin{corollary}
For any seven-note scale $X$, there exist exactly five distinct triadic orbit covers up to scalar translation, corresponding to the five rotation classes in Theorem~\ref{thm:triadic-classes-seven}.
\end{corollary}

\section{Topology of Scale Covers}
\label{sec:topology-of-scale-covers}

Given a covering of a set, we can construct its \emph{nerve}, which is a topological object that encodes the intersection structure of the covering. These nerves often reveal important structural properties. For example, Mazzola has shown \cite[Chapter 13]{mazzola2002topos} that the nerve of the triadic cover of a diatonic scale has the topology of a Möbius strip. One of the most musically salient features of a nerve is its number of $p$-dimensional holes: for instance, the Möbius strip has a single 1-dimensional hole, whereas the nerve of the cover generated by the interval composition $(1,4,2)$ over a seven-note scale has eight 1-dimensional holes.

As we will see in Theorem \ref{thm:heptatonic-triadic-isoclasses}, $k$-adic orbit covers of $n$-note scales can be classified up to isomorphism by the simplicial isomorphism classes of their nerves. This classification is not merely abstract: it captures the underlying intersection patterns of the $k$-ads in the cover. For example, we will see that the tertian orbit cover $X^{((2,2,3))}$ (triads built from thirds) and the quartal orbit cover $X^{((3,3,1))}$ (triads built from fourths) over the diatonic scale have isomorphic nerves. This reflects the fact that their chordal structures share the same intersection properties. In the descending fifths sequence of diatonic triads—CEG, FAC, BDF, EGB, ACE, DFA, GBD, CEG—each adjacent pair shares exactly one common tone, and no three consecutive triads have a tone in common. The same is true of the ascending seconds sequence in quartal harmony—CFB, DGC, EAD, FBE, GCF, ADG, BEA, CFB. Orbit covers with isomorphic nerves therefore afford equivalent common-tone structures and harmonic progressions, even when built from different interval compositions.

\subsection{Nerve Complexes and Orbit Covers}

\begin{definition}[Nerve]
Given a cover $\mathcal{C}_X = \{ C_i \}_{i \in I}$ of a set $X$, its \emph{nerve} $N(\mathcal{C}_X)$ is the simplicial complex whose vertices are indexed by $I$, and whose simplices correspond to nonempty intersections:
\[
N(\mathcal{C}_X) \coloneqq \left\{\, J \subseteq I \mid \bigcap_{j \in J} C_j \neq \varnothing\, \right\}.
\]
The \emph{dimension} of a simplex $J$ is $|J| - 1$, and the dimension of the nerve is the maximum dimension among its simplices. We denote the set of $k$-dimensional simplices by $N(\mathcal{C}_X)_k$.
\end{definition}

If $f : X \to Y$ is a bijection, define the induced cover $\mathcal{C}_{f(X)} \coloneqq \{ f(C_i) \}_{i \in I}$. This yields an isomorphism of nerves
\[
\begin{matrix}
f_\mathcal{C} : & N(\mathcal{C}_X) & \xrightarrow{\;\cong\;} & N(\mathcal{C}_{f(X)}) \\
\\
& J & \longmapsto & J,
\end{matrix}
\]
where the indexing sets are preserved and the intersections are mapped under $f$. More explicitly, for each $j \in J$,
\[
\bigcap_{j \in J} f(C_j) = f\left( \bigcap_{j \in J} C_j \right).
\]

Now let $X^{(A)}$ be a primitive orbit cover with $|A| = k$. If we associate to each set $C = \T_i(A) \in X^{(A)}$ a unique point $c \in C$, then each $x \in X$ lies in exactly $k$ such sets, namely the $k$ translates $\T_j(A)$ containing $x$. These $k$ sets form a $(k{-}1)$-simplex $\Delta_x$ in the nerve $N(X^{(A)})$, uniquely determined by $x$. We call $\Delta_x$ the \emph{harmonic region} associated to $x$. Thus, the harmonic regions correspond bijectively to elements of the scale $X$.

We summarize the key properties of the nerve of a primitive orbit cover:

\begin{proposition}
\label{prop:nerve-properties}
Let $X^{(A)}$ be a primitive orbit cover with $|X| = n$ and $|A| = k$. Then:
\begin{itemize}
    \item $N(X^{(A)})$ is $(k{-}1)$-dimensional;
    \item It contains exactly $n$ distinct $(k{-}1)$-simplices $\Delta_x$, namely the \emph{harmonic regions};
    \item Every $j$-simplex with $j < k-1$ is a face of some $\Delta_x$.
\end{itemize}
\end{proposition}

To canonically enumerate these simplices, we define a \emph{rooted} primitive orbit cover. Let $X$ be a mode with group structure $(X, \oplus)$ and tonic $t \in X$. Fix $\sigma \in \Sigma(n,k)$ and set $A = \sigma(t)$. Then $X^{(A)}$ is said to be \emph{rooted at $t$}. Each translated $k$-ad $\sigma(t \oplus i)$ corresponds to a unique harmonic region $\Delta_{t \oplus i}$. This yields a bijection:
\begin{equation}
\begin{matrix}
\Delta : & X^{(A)} & \xlongrightarrow{\cong} & N(X^{(A)})_{k-1} \\
\\
 & \sigma(x) & \longmapsto & \Delta_x
\end{matrix}
\end{equation}

\subsection{Affine Classification of Nerves}

We now classify orbit covers up to isomorphism of their nerve complexes. While Theorem~\ref{thm:triadic-classes-seven} identifies five distinct triadic orbit covers of a seven-note scale, some are topologically equivalent—i.e., their nerves are isomorphic.

We first note the following fact:
\begin{proposition}
For any affine automorphism $f : \Z_n \to \Z_n$ of the form $f(x) = ux + v$ with $u \in \Z_n^\times$, we have
\[
f \circ \T_i = \T_{ui} \circ f.
\]
\end{proposition}

\begin{proof}
For any $x \in \Z_n$,
\begin{align*}
(f \circ \T_i)(x) & = f(x+i) \\
& = u(x+i) + v \\
& = ux + ui + v \\
& = \T_{ui}(ux + v) \\
& = (\T_{ui} \circ f)(x). \\
\end{align*}
\end{proof}

For $\Z_n^{(\alpha)}$ a primitive orbit cover generated by $\alpha \in \Sigma(n,k)$, we show that $f$ sends $\Z_n^{(\alpha)}$ to another primitive orbit cover $\Z_n^{(\beta)}$ for a unique $\beta \in \Sigma(n,k)$ determined by $\alpha$ and $u$.

\begin{definition}[$u$-transform]
Let $\sigma = (a_1, \ldots, a_k) \in \Sigma(n,k)$ and $u \in \Z_n^\times$. Define partial sums $S_j = \sum_{i=1}^j a_i$ (with $S_0 = 0$), and let 
\[
U_j = \begin{cases}
	u S_j \pmod n & \text{if } S_j < n \\ 
	n  & \text{if } S_j = n
	\end{cases}
\]
After sorting the values so that $0 = U_{j_0} < U_{j_1} < \cdots < U_{j_k} = n$, define
\[
u \cdot \sigma \coloneqq (U_{j_1} - U_{j_0}, \ldots, U_{j_k} - U_{j_{k-1}}) \in \Sigma(n,k)
\]
as the \emph{$u$-transform} of $\sigma$. 
\end{definition}

\begin{proposition}
\label{prop:affine-transport}
Let $f(x) = ux + v$ be an affine automorphism on $\Z_n$, and let $\alpha \in \Sigma(n,k)$ define a primitive orbit cover. Then:
\[
f(\Z_n^{(\alpha)}) = \Z_n^{(u \cdot \alpha)},
\]
and $f$ induces an isomorphism of the associated nerve complexes.
\end{proposition}

\begin{proof}[Sketch]
The image of $\alpha(x)$ under $f$ is a $k$-ad with interval structure given by multiplying the partial sums of $\alpha$ by $u \in \Z_n^\times$, modulo $n$. Sorting and taking differences yields $u \cdot \alpha$.
\end{proof}

\begin{corollary}
Let $\alpha, \beta \in \Sigma(n,k)$ define primitive orbit covers. Then $N(\Z_n^{(\alpha)}) \cong N(\Z_n^{(\beta)})$ if and only if $\beta = u \cdot \alpha$ for some $u \in \Z_n^\times$.
\end{corollary}

This leads to the classification result:

\begin{theorem}[Isomorphism Classes of Triadic Orbit Covers for Heptatonic Scales]
\label{thm:heptatonic-triadic-isoclasses}
Let $\Sigma(7,3)/\mathbf{R}\Z^3$ denote the five rotation classes of triadic interval compositions. Under the action of $\Z_7^\times$ (units modulo 7), these fall into two affine orbits:
\[
\begin{aligned}
\mathcal{O}_1 &= \{ [(1,1,5)],\ [(1,3,3)],\ [(2,2,3)] \}, \\
\mathcal{O}_2 &= \{ [(1,2,4)],\ [(1,4,2)] \}.
\end{aligned}
\]
Hence, there are exactly two nerve isomorphism classes of triadic orbit covers of heptatonic scales.
\end{theorem}

\begin{example}[Isomorphisms of a Diatonic Triadic Cover]
Figure~\ref{fig:triadic-cover-comparison} shows the opening of Bach's chorale BWV 254 in F major, harmonized using the standard triadic orbit cover $F^{((2,2,3))}$ over the diatonic scale
\[
F = \{5, 7, 9, 10, 0, 2, 4\} \subseteq \Z_{12}.
\]

\begin{figure}[ht]
  \centering
  \begin{subfigure}[b]{1\textwidth}
    \centering
    \includegraphics[width=\textwidth]{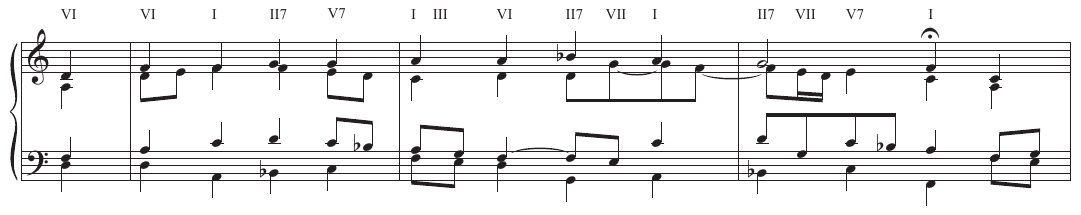}
    \caption{Excerpt harmonized using the diatonic triadic cover \( F^{((2,2,3))} \).}
    \label{fig:diatonic-cover}
  \end{subfigure}
  \vspace{0.75em} 

  \hfill
  \begin{subfigure}[b]{1\textwidth}
    \centering
    \includegraphics[width=\textwidth]{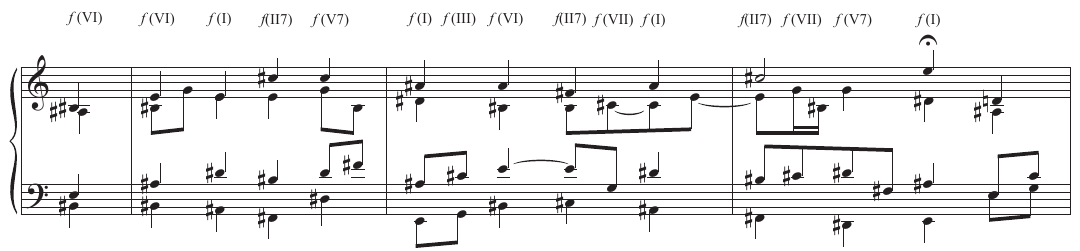}
    \caption{Same excerpt transformed to the exotic cover \( X^{((3,3,1))} \).}
    \label{fig:exotic-cover}
  \end{subfigure}
  \caption{Comparison of harmonic structure in diatonic and exotic triadic orbit covers. The transformation preserves orbit structure and local common-tone relations.}
  \label{fig:triadic-cover-comparison}
\end{figure}

Now let $X = \{4, 6, 7, 10, 0, 1, 3\}$ be a non-diatonic scale. Consider the affine automomorphism $f(x) = 5x \pmod 7$ on $\Z_7$, which maps the interval composition $(2,2,3)$ to $(3,3,1)$, i.e., 
\[
f \cdot (2,2,3) = (3,3,1).
\] 
This induces an isomorphism between the triadic cover $F^{((2,2,3))}$ and the orbit cover
\[
X^{((3,3,1))} = \{\{4, 10, 3\}, \{6, 0, 4\}, \{7, 1, 6\}, \{10, 3, 7\}, \{0, 4, 10\}, \{1, 6, 0\}, \{3, 7, 1\}\}.
\]
One such isomorphism between the two covers is realized by the following pointwise correspondence:
\[
\begin{aligned}
5 &\mapsto 4, \\
7 &\mapsto 1, \\
9 &\mapsto 10, \\
10 &\mapsto 6, \\
0 &\mapsto 3, \\
2 &\mapsto 0, \\
4 &\mapsto 7.
\end{aligned}
\]

Despite the fact that the harmonic content of $X^{((3,3,1))}$ differs sharply from the diatonic triads in $F^{((2,2,3))}$—featuring chords of set-class types such as 016 and 026—the underlying common-tone structure is preserved. For example, the I–VI progression in BWV 254 (F major to D minor) shares two common tones, and its image in $X^{((3,3,1))}$ does as well. This shared nerve structure gives the transformed cover a perceptual coherence akin to functional harmony, even as its harmonic surface sounds more post-tonal and exotic.
\end{example}

\section{Outlook and Future Prospects}

The author is currently composing a series of works titled \emph{Chorales}, which draw directly on the theoretical framework presented in this paper. These compositions make extensive use of both the scale structures and orbit covers introduced here, but they also rely on a broader and more sophisticated framework the author is developing in \emph{Mathematical Principles of a Generalized Tonal Practice}. This framework takes the core constructions of this paper as foundational, and extends them to support a fully generalized theory of tonality.

This developing theory of tonality enables the formulation of functional harmonic syntax in the spirit of Rohrmeier’s generative approach \cite{rohrmeier2011towards}, while expanding it beyond the diatonic domain to encompass a wide variety of orbit covers and scales. In doing so, it supports a generalized tonal practice that can interface with both traditional and post-tonal harmonic idioms.

Future work will present this more comprehensive theory of tonal systems, along with compositional applications. 

\pagebreak

\printbibliography

\end{document}